\documentclass[12pt]{amsart}

\usepackage{amsmath}
\usepackage{amssymb}
\usepackage{amsfonts}
\usepackage{amsthm}
\usepackage{enumerate}
\usepackage{hyperref}
\usepackage{color}

\theoremstyle{plain}
\newtheorem{thm}{Theorem}[section]

\newtheorem{lem}[thm]{Lemma}
\newtheorem{cor}[thm]{Corollary}

\newtheorem{ques}[thm]{Question}

\theoremstyle{definition}
\newtheorem{dfn}[thm]{Definition}

\newtheorem{dfns-rems}[thm]{Definitions and Remarks}
\newtheorem{notas-rems}[thm]{Notations and Remarks}
\newtheorem{exmps-rems}[thm]{Examples and Remarks}

\begin{document}


\title[Depth and sdepth of symbolic powers of cover ideals]{Depth and Stanley depth of symbolic powers of cover ideals of graphs}


\author[S. A. Seyed Fakhari]{S. A. Seyed Fakhari}

\address{S. A. Seyed Fakhari, School of Mathematics, Statistics and Computer Science,
College of Science, University of Tehran, Tehran, Iran.}

\email{aminfakhari@ut.ac.ir}

\urladdr{http://math.ipm.ac.ir/$\sim$fakhari/}


\begin{abstract}
Let $G$ be a graph with $n$ vertices and let $S=\mathbb{K}[x_1,\dots,x_n]$ be the
polynomial ring in $n$ variables over a field $\mathbb{K}$. Assume that $J(G)$ is the cover ideal of $G$ and $J(G)^{(k)}$ is its $k$-th symbolic power. We prove that the sequences $\{{\rm sdepth}(S/J(G)^{(k)})\}_{k=1}^\infty$ and $\{{\rm sdepth}(J(G)^{(k)})\}_{k=1}^\infty$ are non-increasing and hence convergent. Suppose that $\nu_{o}(G)$ denotes the ordered matching number of $G$. We show that for every integer $k\geq 2\nu_{o}(G)-1$, the modules $J(G)^{(k)}$ and $S/J(G)^{(k)}$ satisfy the Stanley's inequality. We also provide an alternative proof for \cite[Theorem 3.4]{hktt} which states that ${\rm depth}(S/J(G)^{(k)})=n-\nu_{o}(G)-1$, for every integer $k\geq 2\nu_{o}(G)-1$.
\end{abstract}


\subjclass[2000]{Primary: 13C15, 05E99; Secondary: 13C13}


\keywords{Stanley depth, Cover ideal, Symbolic power, Ordered matching number}


\thanks{}


\maketitle


\section{Introduction} \label{sec1}

Let $S=\mathbb{K}[x_1,\dots,x_n]$ be the
polynomial ring in $n$ variables over a field $\mathbb{K}$ and suppose that $M$ is a nonzero finitely generated $\mathbb{Z}^n$-graded $S$-module. Let
$u\in M$ be a homogeneous element and $Z\subseteq
\{x_1,\dots,x_n\}$. The $\mathbb {K}$-subspace $u\mathbb{K}[Z]$
generated by all elements $uv$ with $v\in \mathbb{K}[Z]$ is
called a {\it Stanley space} of dimension $|Z|$, if it is a free
$\mathbb{K}[Z]$-module. Here, as usual, $|Z|$ denotes the number
of elements of $Z$. A decomposition $\mathcal{D}$ of $M$ as a
finite direct sum of Stanley spaces is called a {\it Stanley
decomposition} of $M$. The minimum dimension of a Stanley space
in $\mathcal{D}$ is called the {\it Stanley depth} of
$\mathcal{D}$ and is denoted by ${\rm sdepth} (\mathcal {D})$.
The quantity $${\rm sdepth}(M):=\max\big\{{\rm sdepth}
(\mathcal{D})\mid \mathcal{D}\ {\rm is\ a\ Stanley\
decomposition\ of}\ M\big\}$$ is called the {\it Stanley depth}
of $M$. As a convention, we set ${\rm sdepth}(M)=\infty$, when $M$ is the zero module. We say that a $\mathbb{Z}^n$-graded $S$-module $M$ satisfies the {\it Stanley's inequality} if $${\rm depth}(M) \leq
{\rm sdepth}(M).$$ In fact, Stanley \cite{s} conjectured that every $\mathbb{Z}^n$-graded $S$-module satisfies the Stanley's inequality.
We refer to \cite{psty} for a reader friendly introduction to Stanley depth, and to \cite{h} for a nice survey on this topic.

The Stanley's conjecture has been recently disproved in \cite{abcj}. The counterexample presented in \cite{abcj} lives in the category of squarefree monomial ideals. Thus, one can still ask whether the Stanley's inequality holds for non-squarefree monomial ideals. Of particular interest are the high powers of monomial ideals. In other words, we ask following question.

\begin{ques} \label{q1}
Let $I$ be a monomial ideal. Is it true that $I^k$ and $S/I^k$ satisfy the Stanley's inequality for every integer $k\gg 0$?
\end{ques}

We approached this question for edge ideals in \cite{asy}, \cite{psy} and \cite{s6}. The most general results are obtained in \cite{s6}. In that paper, we proved that if $G$ is a graph with $n$ vertices and $I(G)$ is its edge ideal, then $S/I(G)^k$ satisfies the Stanley's inequality for every integer $k\geq n-1$ \cite[Corollary 2.5]{s6}. If moreover $G$ is a non-bipartite graph, or
at least one of the connected components of $G$ is a tree with at least one edge, then $I(G)^k$ satisfies the Stanley's inequality for every integer $k\geq n-1$ \cite[Corollary 3.6]{s6}.

Recently, in \cite{s4}, we studied the powers of cover ideal of bipartite graphs. We proved in \cite[Corollary 3.5]{s4} that if $G$ is a bipartite graph with cover ideal $J(G)$, then $J(G)^k$ and $S/J(G)^k$ satisfy the Stanley`s conjecture for $k\gg 0$. On the other hand, we know from \cite[Corollary 2.6]{grv} that for every bipartite graph $G$ and every integer $k\geq 1$, we have $J(G)^k=J(G)^{(k)}$, where $J(G)^{(k)}$ denotes the $k$-th symbolic power of $J(G)$. Hence,  \cite[Corollary 3.5]{s4} essentially says that for any bipartite graph $G$, the modules $J(G)^{(k)}$ and $S/J(G)^{(k)}$ satisfy the Stanley's inequality for every integer $k\gg 0$. In this regard, we ask an analogue of Question \ref{q1} for symbolic powers.

\begin{ques} \label{q2}
Let $I$ be a squarefree monomial ideal. Is it true that $I^{(k)}$ and $S/I^{(k)}$ satisfy the Stanley's inequality for every integer $k\gg 0$?
\end{ques}

In this  paper, we give a positive answer to Question \ref{q2}, in the case that $I=J(G)$ is the cover ideal of a graph. We mention that the depth of symbolic powers of cover ideals has been studied by Hoa et al. \cite{hktt}. Using a result of Hoa and Trung \cite{ht2}, the authors of \cite{hktt} notice that the depth of symbolic powers of a squarefree monomial ideal is stable. In the case of cover ideal of graphs, they prove that the limit value can be combinatorially described. In fact, they prove the following stronger result. Let $G$ be a graph with $n$ vertices and let $\nu_{o}(G)$ denote the ordered matching number of $G$ (see Definition \ref{om}). It is shown in \cite[Theorem 3.4]{hktt} that for every integer $k\geq 2\nu_{o}(G)-1$, we have ${\rm depth}(S/J(G)^{(k)})=n-\nu_{o}(G)-1$. In Theorem \ref{ldepth}, we provide an alternative proof for this result. Indeed, the proof in \cite{hktt} is based on a formula due to Takayama \cite[Theorem 2.2]{t}, while we do not use Takayama's formula. Next, we consider the sequences $\{{\rm sdepth}(S/J(G)^{(k)})\}_{k=1}^\infty$ and $\{{\rm sdepth}(J(G)^{(k)})\}_{k=1}^\infty$. We prove in Theorem \ref{sdepthsym} that theses sequences are non-increasing and thus, convergent. Unfortunately, we are not able to determine the precise value of limit of theses sequences. However, we prove in Theorem \ref{main} that for every integer $k\geq 0$, we have$${\rm sdepth}(J(G)^{(k)})\geq n-\nu_{o}(G) \ \ \ \ {\rm and} \ \ \ \ {\rm sdepth}(S/J(G)^{(k)})\geq n-\nu_{o}(G)-1.$$In other words, we have the following lower bounds for the limit of the above mentioned sequences.$$\lim_{k\rightarrow\infty}{\rm sdepth}(J(G)^{(k)})\geq n-\nu_{o}(G) \ \ \ \  {\rm and} \ \ \ \  \lim_{k\rightarrow\infty}{\rm sdepth}(S/J(G)^{(k)})\geq n-\nu_{o}(G)-1.$$

Finally, we conclude from Theorems \ref{ldepth} and \ref{main} that for every graph $G$ and every integer $k\geq 2\nu_{o}(G)-1$, the modules $J(G)^{(k)}$ and $S/J(G)^{(k)}$ satisfy the Stanley's inequality.


\section{Preliminaries} \label{sec1'}

In this section, we provide the definitions and basic facts which will be used in the next section.

Let $G$ be a graph with vertex set $V(G)=\big\{x_1, \ldots,
x_n\big\}$ and edge set $E(G)$ (by abusing the notation, we identify the vertices of $G$ with the variables of $S$). For a vertex $x_i$, the {\it neighbor set} of $x_i$ is $N_G(x_i)=\{x_j\mid x_ix_j\in E(G)\}$ and we set $N_G[x_i]=N_G(x_i)\cup \{x_i\}$ and call it the {\it closed neighborhood} of $x_i$. For every subset $A\subset V(G)$, the graph $G\setminus A$ is the graph with vertex set $V(G\setminus A)=V(G)\setminus A$ and edge set $E(G\setminus A)=\{e\in E(G)\mid e\cap A=\emptyset\}$. A subgraph $H$ of $G$ is called {\it induced} provided that two vertices of $H$ are adjacent if and only if they are adjacent in $G$. A {\it matching} in $G$ is a subgraph consisting of pairwise disjoint edges. If the subgraph is an induced subgraph, the matching is an {\it induced matching}. The cardinality of the maximum induced matching of $G$ is denoted by ${\rm indmatch}(G)$. A subset $W$ of $V(G)$ is called an {\it independent subset} of $G$ if there are no edges among the vertices of $W$. A subset $C$ of $V(G)$ is called a {\it vertex cover} of the graph $G$ if every edge of $G$ is incident to at least one vertex of $C$. A vertex cover $C$ is called a {\it minimal vertex cover} of $G$ if no proper subset of $C$ is a vertex cover of $G$.

Next, we define the notion of ordered matching for a graph. It was introduced in \cite{cv} and plays a key role in this paper.

\begin{dfn} \label{om}
Let $G$ be a graph, and let $M=\{\{a_i,b_i\}\mid 1\leq i\leq r\}$ be a
nonempty matching of $G$. We say that $M$ is an {\it ordered matching} of
$G$ if the following hold:
\begin{itemize}
\item[(1)] $A:=\{a_1,\ldots,a_r\} \subseteq V(G)$ is a set of
    independent vertices of $G$; and

\item[(2)] $\{a_i, b_j\}\in E(G)$ implies that $i\leq j$.
\end{itemize}
The {\it ordered matching number} of $G$, denoted by $\nu_{o}(G)$, is
defined to be $$\nu_{o}(G)=\max\{|M|\mid M\subseteq E(G)\ {\rm is\ an\
ordered\ matching\ of} \ G\}.$$
\end{dfn}

The edge ideal $I(G)$ of $G$ is the ideal of $S$ generated by the squarefree  monomials  $x_ix_j$, where $\{x_i, x_j\}$ is an edge of $G$. The Alexander dual of the edge ideal of $G$ in $S$, i.e., the
ideal $$J(G)=I(G)^{\vee}=\bigcap_{\{x_i,x_j\}\in E(G)}(x_i,x_j),$$ is called the
{\it cover ideal} of $G$ in $S$. The reason for this name is due to the
well-known fact that the generators of $J(G)$ correspond to minimal vertex covers of $G$.

Let $I$ be an ideal of $S$ and let ${\rm Min}(I)$ denote the set of minimal primes of $I$. For every integer $k\geq 1$, the $k$-th {\it symbolic power} of $I$,
denoted by $I^{(k)}$, is defined to be$$I^{(k)}=\bigcap_{\frak{p}\in {\rm Min}(I)} {\rm Ker}(R\rightarrow (R/I^k)_{\frak{p}}).$$Let $I$ be a squarefree monomial ideal in $S$ and suppose that $I$ has the irredundant
primary decomposition $$I=\frak{p}_1\cap\ldots\cap\frak{p}_r,$$ where every
$\frak{p}_i$ is an ideal of $S$ generated by a subset of the variables of
$S$. It follows from \cite[Proposition 1.4.4]{hh} that for every integer $k\geq 1$, $$I^{(k)}=\frak{p}_1^k\cap\ldots\cap
\frak{p}_r^k.$$

Let $M$ be a finitely generated graded $S$-Module. The {\it projective dimension} and the {\it Castelnuovo-Mumford regularity} (or simply, {\it regularity}) of $M$, are defined as follows.
$${\rm pd}(M)=\max\{i|\ {\rm Tor}_i^S(\mathbb{K}, M)\neq0\}, \ \ \ \ {\rm reg}(M)=\max\{j-i|\ {\rm Tor}_i^S(\mathbb{K}, M)_j\neq0\}.$$


\section{Main results} \label{sec2}

Let $G$ be a graph with $n$ vertices and let $J(G)$ denote its cover ideal. It is shown in \cite[Theorem 3.4]{hktt} that for every integer $k\geq 2\nu_{o}(G)-1$, we have$${\rm depth}(S/J(G)^{(k)})=n-\nu_{o}(G)-1.$$In Theorem \ref{ldepth}, we provide an alternative proof for this result. We use the method of polarization in our proof. Hence, we briefly review this method for the sake of completeness.

Let $I$ be a monomial ideal of
$S$ with minimal generators $u_1,\ldots,u_m$,
where $u_j=\prod_{i=1}^{n}x_i^{a_{i,j}}$, $1\leq j\leq m$. For every $i$
with $1\leq i\leq n$, let $a_i=\max\{a_{i,j}\mid 1\leq j\leq m\}$, and
suppose that $$T=\mathbb{K}[x_{1,1},x_{1,2},\ldots,x_{1,a_1},x_{2,1},
x_{2,2},\ldots,x_{2,a_2},\ldots,x_{n,1},x_{n,2},\ldots,x_{n,a_n}]$$ is a
polynomial ring over the field $\mathbb{K}$. Let $I^{{\rm pol}}$ be the squarefree
monomial ideal of $T$ with minimal generators $u_1^{{\rm pol}},\ldots,u_m^{{\rm pol}}$, where
$u_j^{{\rm pol}}=\prod_{i=1}^{n}\prod_{k=1}^{a_{i,j}}x_{i,k}$, $1\leq j\leq m$. The ideal $I^{{\rm pol}}$
is called the {\it polarization} of $I$. We know from \cite[Corollary 1.6.3]{hh} that polarization preserves the projective dimension, i.e.,$${\rm pd}(S/I)={\rm pd}(T/I^{{\rm pol}}).$$

\begin{thm} \label{ldepth}
Let $G$ be a graph with $n$ vertices. Then for every integer $k\geq 2\nu_{o}(G)-1$, we have ${\rm depth}(S/J(G)^{(k)})=n-\nu_{o}(G)-1$.
\end{thm}

\begin{proof}
Without loss of generality suppose that $G$ has no isolated vertex. Combining \cite[Proposition 2.4]{v} and \cite[Theorem 2.8]{cv}, we have$${\rm depth}(S/J(G)^{(k)})\geq n-\nu_{o}(G)-1,$$for every integer $k\geq 1$. Thus, we prove that$${\rm depth}(S/J(G)^{(k)})\leq n-\nu_{o}(G)-1,$$for every integer $k\geq 2\nu_{o}(G)-1$. Assume that $k\geq 2\nu_{o}(G)-1$ is an integer and suppose that $V(G)=\{x_1, \ldots, x_n\}$ is the vertex set of $G$. Using the Auslander-Buchsbaum formula, we need to show that ${\rm pd}(S/J(G)^{(k)})\geq \nu_{o}(G)+1$. As polarization preserves the projective dimension, it is enough to prove that$${\rm pd}(T/(J(G)^{(k)})^{{\rm pol}})\geq \nu_{o}(G)+1,$$where $T$ is a new polynomial ring containing $(J(G)^{(k)})^{{\rm pol}}$. We know from \cite[Lemma 3.4]{s3} that $(J(G)^{(k)})^{{\rm pol}}$ is the cover ideal of a new graph, say $G_k$, with the vertex set $V(G_k)=\{x_{i,p}\mid 1\leq i\leq n,  1\leq p\leq k\}$ and the edge set$$E(G_k)=\{\{x_{i,p}, x_{j,q}\}\mid \{x_i, x_j\}\in E(G) \  {\rm and} \  p+q\leq k+1\}.$$ Using Terai's theorem \cite[Theorem 8.1.10]{hh}, it is sufficient to show that $${\rm reg}(T/(I(G_k))\geq \nu_{o}(G).$$ By \cite[Lemma 2.2]{k}, we know that$${\rm reg}(T/(I(G_k))\geq {\rm indmatch}(G_k).$$We show that ${\rm indmatch}(G_k)\geq \nu_{o}(G)$ and this completes the proof. Set $t=\nu_{o}(G)$ and assume without loss of generality that $\big\{\{x_i,x_{t+i}\}\mid 1\leq i\leq t\big\}$ is an ordered matching of $G$ so that
\begin{itemize}
\item[$\bullet$] $\{x_1, \ldots, x_t\}$ is an independent subset of
     vertices of $G$.
\item[$\bullet$] $\{x_i, x_{t+j}\}\in E(G)$ implies that $i\leq j$.
\end{itemize}

We show that the set of edges$$\big\{\{x_{i,t+1-i}, x_{t+i, k+i-t}\}\mid 1\leq i\leq t\big\}$$is an induced matching of $G_k$. Indeed, fix the integers $i$ and $j$ with $1\leq i <j\leq t$. As $\{x_1, \ldots, x_t\}$ is an independent subset of
vertices of $G$, it follows from the construction of $G_k$ that the vertices $x_{i,t+1-i}$ and $x_{j,t+1-j}$ are not adjacent in $G_k$. Since, $i<j$, we conclude that $\{x_j, x_{t+i}\}\notin E(G)$ and it again follows from the construction of $G_k$ that $\{x_{j,t+1-j}, x_{t+i, k+i-t}\}\notin E(G_k)$. On the other hand,$$(t+1-i)+(k+j-t)=k+j-i+1>k+1,$$which shows that $\{x_{i,t+1-i}, x_{t+j, k+j-t}\}$ is not an edge of $G_k$. Finally, since $k\geq 2t-1$, $i\geq 1$ and $j\geq 2$, we have$$(k+i-t)+(k+j-t)=k+(k-2t)+i+j\geq k+2.$$Thus $\{x_{t+i,k+i-�t}, x_{t+j,k+j-t}\}\notin E(Gk)$.

Therefore, $$\big\{\{x_{i,t+1-i}, x_{t+i, k+i-t}\}\mid 1\leq i\leq t\big\}$$is an induced matching of $G_k$. Hence, ${\rm indmatch}(G_k)\geq t=\nu_{o}(G)$.
\end{proof}

Let $G$ be a graph. It is shown in \cite[Theorem 3.2]{hktt} that the depth of symbolic powers of $J(G)$ is a non-increasing function. Our next goal is to prove a similar result for the Stanley depth (see Theorem \ref{sdepthsym}). In order to do this, we need to recall the notion of the Stanley regularity.

Let $M$ be a finitely generated $\mathbb{Z}^n$-graded $S$-module and assume that$$\mathcal{D} : M=\bigoplus_{i=1}^m u_i \mathbb{K}[Z_i]$$is a Stanley decomposition of $M$. Following \cite{so}, we define the {\it Stanley regularity} of $\mathcal{D}$ to be ${\rm sreg}(\mathcal{D})=\max\{{\rm deg}(u_i): 1=1, \ldots, m\}$. Also the {\it Stanley regularity} of $M$ is defined as$${\rm sreg}(M):=\min\big\{{\rm sreg}
(\mathcal{D})\mid \mathcal{D}\ {\rm is\ a\ Stanley\ decomposition\ of}\ M\big\}.$$For every graph $G$ with $n$ vertices, we know from \cite[Corollary 46]{h} that
\[
\begin{array}{rl}
{\rm sreg}(I(G))+{\rm sdepth}(S/J(G))=n \ \ \ {\rm and} \ \ \ {\rm sdepth}(J(G))+{\rm sreg}(S/I(G))=n.
\end{array} \tag{$\dagger$} \label{dag}
\]

\begin{lem} \label{del}
Assume  that $I$ is a monomial ideal of $S=\mathbb{K}[x_1,\ldots,x_n]$. Let $S'=\mathbb{K}[x_1, \ldots, x_{j-1}, x_{j+1}, \ldots x_n]$ be the polynomial ring obtained from $S$ by deleting the variable $x_j$ and consider the ideal $I'=I\cap S'$. Then

\begin{itemize}
\item[(i)] ${\rm sreg}(I')\leq {\rm sreg}(I)$.
\item[(ii)] ${\rm sreg}(S'/I')\leq {\rm sreg}(S/I)$.
\end{itemize}
\end{lem}

\begin{proof}
(i) consider a Stanley decomposition $$\mathcal{D} : I=\bigoplus_{i=1}^m u_i \mathbb{K}[Z_i]$$ of $I$, such that ${\rm sreg}(\mathcal{D})={\rm sreg}(I)$. Without loss of generality, we may assume that there exist an integer $t$ with $0\leq t \leq m $ such that $u_1, \ldots, u_t$ are divisible by $x_j$ and $u_{t+1}, \ldots, u_m$ are not divisible by $x_j$. Then $\mathcal{D}' : I'=\bigoplus_{i=t+1}^m u_i \mathbb{K}[Z_i\setminus \{x_j\}]$ is a Stanley decomposition of $I'$. Then ${\rm sreg}(I')\leq {\rm sreg}(\mathcal{D}')\leq {\rm sreg}(I)$.

The proof of (ii) is similar.
\end{proof}

We are now ready to prove that the Stanley depth of symbolic powers of cover ideals is a non-increasing function.

\begin{thm} \label{sdepthsym}
Let $G$ be a graph. Then for every integer $k\geq 1$, we have

\begin{itemize}
\item[(i)] ${\rm sdepth}(S/J(G)^{(k)})\geq {\rm sdepth}(S/J(G)^{(k+1)}).$
\item[(ii)] ${\rm sdepth}(J(G)^{(k)})\geq {\rm sdepth}(J(G)^{(k+1)}).$
\end{itemize}
In particular, the sequences $\{{\rm sdepth}(S/J(G)^{(k)})\}_{k=1}^\infty$ and $\{{\rm sdepth}(J(G)^{(k)})\}_{k=1}^\infty$ are convergent.
\end{thm}

\begin{proof}
(i) We may assume that $G$ has no isolated vertex. Suppose that $V(G)=\{x_1, \ldots, x_n\}$ is the vertex set of $G$. We use the method of polarization. Thus, consider the ideals $(J(G)^{(k)})^{{\rm pol}}\subseteq T_k$ and $(J(G)^{(k+1)})^{{\rm pol}}\subseteq T_{k+1}$, where$$T_r=\mathbb{K}[x_{i,p}\mid 1\leq i\leq n,  1\leq p\leq r].$$By \cite[Corollary 4.4]{ikm}, we know that$${\rm sdepth}(T_r/(J(G)^{(r)})^{{\rm pol}})={\rm sdepth}(S/J(G)^{(r)})+n(r-1).$$ Hence, we must prove that$${\rm sdepth}(T_k/(J(G)^{(k)})^{{\rm pol}})\geq {\rm sdepth}(T_{k+1}/(J(G)^{(k+1)})^{{\rm pol}})-n.$$It follows from \cite[Lemma 3.4]{s3} that for $r=k, k+1$ the ideal $(J(G)^{(r)})^{{\rm pol}}$ is the cover ideal of a new graph, say $G_r$, with the vertex set $V(G_r)=\{x_{i,p}\mid 1\leq i\leq n,  1\leq p\leq r\}$ and the edge set$$E(G_r)=\{\{x_{i,p}, x_{j,q}\}\mid \{x_i, x_j\}\in E(G) \  {\rm and} \  p+q\leq r+1\}.$$ Therefore, we show that $${\rm sdepth}(T_k/J(G_k))\geq {\rm sdepth}(T_{k+1}/J(G_{k+1}))-n.$$Using equalities (\ref{dag}), it is sufficient to prove that$${\rm sreg}(I(G_k))\leq {\rm sreg}(I(G_{k+1})).$$

{\bf Claim.} $G_k$ is an induced subgraph of $G_{k+1}$.

{\it Proof of the claim.} Set$$W:=\{x_{i, \lfloor\frac{k+1}{2}\rfloor +1} \mid 1\leq i\leq n\}\subseteq V(G_{k+1}).$$Consider the map $\varphi : V(G_k)\rightarrow V(G_{k+1}\setminus W)$ which is defined as follows.

$$\varphi(x_{i,j}) =
\left\{
	\begin{array}{ll}
		x_{i,j}  & \mbox{if } 1\leq j\leq \lfloor\frac{k+1}{2}\rfloor \\
		x_{i,j+1} & \mbox{if } j\geq  \lfloor\frac{k+1}{2}\rfloor+1
	\end{array}
\right.$$

One can easily check that $\varphi$ induces an isomorphism between $G_k$ and $G_{k+1}\setminus W$. Thus, $G_k$ is an induced subgraph of $G_{k+1}$.

The assertion of (i) now follows from the claim and Lemma \ref{del}.

The proof of (ii) is similar.
\end{proof}

In \cite[Corollary 3.5]{s4}, the author proves that for any bipartite graph $G$, the modules $J(G)^k$ and $S/J(G)^k$ satisfy the Stanley's inequality for every integer $k\gg 0$. On the other hand, we know from \cite[Corollary 2.6]{grv} that for every bipartite graph $G$ and every integer $k\geq 1$, we have $J(G)^k=J(G)^{(k)}$. Thus, \cite[Corollary 3.5]{s4} is indeed saying that for any bipartite graph $G$, the modules $J(G)^{(k)}$ and $S/J(G)^{(k)}$ satisfy the Stanley's inequality for every integer $k\gg 0$. In Corollary \ref{sineq}, we will extend this result to any arbitrary graph $G$. The proof of Corollary \ref{sineq} is based on Theorem \ref{main}, which determines a lower bound for the Stanley depth of symbolic powers of cover ideals.

We first need the following simple lemma.

\begin{lem} \label{colon}
Let $G$ be a graph with $V(G)=\{x_1, \ldots, x_n\}$. Set $u=\prod_{i=1}^nx_i$. Then for every integer $k\geq 2$, we have $(J(G)^{(k)}:u)=J(G)^{(k-2)}$.
\end{lem}

\begin{proof}
Let $k\geq 2$ be an integer. For every couple of integers $1\leq i,j\leq n$ with $i\neq j$, we have $((x_i,x_j)^k:u)=(x_i,x_j)^{k-2}$. Hence,$$(J(G)^{(k)}:u)=\bigcap_{\{x_i, x_j\}\in E(G)}((x_i,x_j)^k:u)=\bigcap_{\{x_i, x_j\}\in E(G)}(x_i,x_j)^{k-2}=J(G)^{(k-2)}.$$
\end{proof}

We are now ready to determine a lower bound for the Stanley depth of symbolic powers of cover ideals.

\begin{thm} \label{main}
Let $G$ be a graph with $n$ vertices and $J(G)$ its cover ideal. Then for every integer $k\geq 1$, we have
$${\rm sdepth}(J(G)^{(k)})\geq n-\nu_{o}(G) \ \ \  {\rm and} \ \ \  {\rm sdepth}(S/J(G)^{(k)})\geq n-\nu_{o}(G)-1.$$
\end{thm}

\begin{proof}
Assume that $V(G)=\{x_1, \ldots, x_n\}$. Let $m$ be the number of edges of $G$. We prove the assertions by induction on $m+k$. First, we can assume that $G$ has no isolated vertex, as deleting the isolated vertices does not change the cover ideal and the ordered matching number of $G$.

If $k=1$, the assertions follow from \cite[Theorem 2.4]{s4}. If $m=1$, then $G$ has two vertices, i.e., $n=2$, and $\nu_{o}(G)=1$. In this case, the second inequality is trivial, while the first inequality follows from \cite[Corollary 24]{h}. Therefore, assume that $k,m\geq 2$. Let $S_1=\mathbb{K}[x_2, \ldots, x_n]$ be the polynomial ring obtained from $S$ by deleting the variable $x_1$ and consider the ideals $J_1=J(G)^{(k)}\cap S_1$ and
$J_1'=(J(G)^{(k)}:x_1)$.

Now $J(G)^{(k)}=J_1\oplus x_1J_1'$ and $S/J(G)^{(k)}=(S_1/J_1)\oplus
x_1(S/J_1')$ (as vector spaces) and therefore by definition of  the Stanley depth we have
\[
\begin{array}{rl}
{\rm sdepth}(J(G)^{(k)})\geq \min \{{\rm sdepth}_{S_1}(J_1), {\rm sdepth}_S(J_1')\},
\end{array} \tag{1} \label{1}
\]
and
\[
\begin{array}{rl}
{\rm sdepth}(S/J(G)^{(k)})\geq \min \{{\rm sdepth}_{S_1}(S_1/J_1), {\rm sdepth}_S(S/J_1')\}.
\end{array} \tag{2} \label{2}
\]

Set $u_1=\prod_{x\in N_G(x_1)}x$ and notice that $J_1=(J(G)\cap S_1)^{(k)}$. Hence, by \cite[Lemma 2.2]{s4}, we conclude that $J_1=u_1^kJ(G\setminus N_G[x_1])^{(k)}S_1$. It follows from \cite[Theorem 1.1]{c1} that ${\rm sdepth}_{S_1}(J_1)={\rm sdepth}_{S_1}(J(G\setminus N_G[x_1])^{(k)}S_1)$ and ${\rm sdepth}_{S_1}(S_1/J_1)={\rm sdepth}_{S_1}(S_1/J(G\setminus N_G[x_1])^{(k)}S_1)$. Therefore, by \cite[Lemma 2.1]{s4} and the induction hypothesis, we conclude that
$${\rm sdepth}_{S_1}(J_1)={\rm sdepth}_{S_1}(J(G\setminus N_G[x_1])^kS_1)\geq n-1-\nu_{o}(G\setminus N_G[x_1])\geq n-\nu_{o}(G),$$
and similarly ${\rm sdepth}_{S_1}(S_1/J_1)\geq n-\nu_{o}(G)-1$. Thus, using the inequalities (\ref{1}) and (\ref{2}), it is enough to prove that ${\rm sdepth}_S(J_1')\geq n-\nu_{o}(G)$ and ${\rm sdepth}_S(S/J_1')\geq n-\nu_{o}(G)-1$.

For every integer $i$ with $2\leq i\leq n$, let $S_i=\mathbb{K}[x_1, \ldots, x_{i-1}, x_{i+1}, \ldots, x_n]$ be the polynomial ring obtained from $S$ by deleting the variable $x_i$ and consider the ideals $J_i'=(J_{i-1}':x_i)$ and $J_i=J_{i-1}'\cap S_i$.

{\bf Claim.} For every integer $i$ with $1\leq i\leq n-1$ we have$${\rm sdepth}(S/J_i')\geq \min\{n-\nu_{o}(G)-1, {\rm sdepth}(S/J_{i+1}')\}$$ and $${\rm sdepth}(J_i')\geq \min \{n-\nu_{o}(G), {\rm sdepth}(J_{i+1}')\}.$$

{\it Proof of the Claim.} For every integer $i$ with $1\leq i\leq n-1$, we have $J_i'=J_{i+1}\oplus x_{i+1}J_{i+1}'$ and $S/J_i'=(S_{i+1}/J_{i+1})\oplus
x_{i+1}(S/J_{i+1}')$ (as vector spaces) and therefore by definition of  the Stanley depth we have
\[
\begin{array}{rl}
{\rm sdepth}(J_i')\geq \min \{{\rm sdepth}_{S_{i+1}}(J_{i+1}), {\rm sdepth}_S(J_{i+1}')\},
\end{array} \tag{3} \label{3}
\]
and
\[
\begin{array}{rl}
{\rm sdepth}(S/J_i')\geq \min \{{\rm sdepth}_{S_{i+1}}(S_{i+1}/J_{i+1}), {\rm sdepth}_S(S/J_{i+1}')\}.
\end{array} \tag{4} \label{4}
\]

Notice that for every integer $i$ with $1\leq i\leq n-1$, we have $J_i'=(J(G)^{(k)}:x_1x_2\ldots x_i)$. Thus$$J_{i+1}=J_i'\cap S_{i+1}=((J(G)^{(k)}\cap S_{i+1}):_{S_{i+1}}x_1x_2\ldots x_i).$$Hence, using \cite[Proposition 2]{p} and \cite[Proposition 2.7]{c}, we conclude that

\[
\begin{array}{rl}
{\rm sdepth}_{S_{i+1}}(J_{i+1})\geq {\rm sdepth}_{S_{i+1}}(J(G)^{(k)}\cap S_{i+1})
\end{array} \tag{5} \label{5}
\]
and
\[
\begin{array}{rl}
{\rm sdepth}_{S_{i+1}}(S_{i+1}/J_{i+1})\geq {\rm sdepth}_{S_{i+1}}(S_{i+1}/(J(G)^{(k)}\cap S_{i+1})).
\end{array} \tag{6} \label{6}
\]

Set $u_{i+1}=\prod_{x\in N_G(x_{i+1})}x$. By \cite[Lemma 2.2]{s4}, $$J(G)\cap S_{i+1}=u_{i+1}J(G\setminus N_G[x_{i+1}])S_{i+1}=\big((u_{i+1})\cap J(G\setminus N_G[x_{i+1}])\big)S_{i+1}.$$ Therefore$$J(G)^{(k)}\cap S_{i+1}=(J(G)\cap S_{i+1})^{(k)}=u_{i+1}^kJ(G\setminus N_G[x_{i+1}])^{(k)}S_{i+1}.$$It follows from \cite[Theorem 1.1]{c1} that$${\rm sdepth}_{S_{i+1}}(J(G)^{(k)}\cap S_{i+1})={\rm sdepth}_{S_{i+1}}(J(G\setminus N_G[x_{i+1}])^{(k)}S_{i+1})$$and$${\rm sdepth}_{S_{i+1}}(S_{i+1}/(J(G)^{(k)}\cap S_{i+1}))={\rm sdepth}_{S_{i+1}}(S_{i+1}/J(G\setminus N_G[x_{i+1}])^{(k)}S_{i+1}).$$ Therefore by \cite[Lemma 2.1]{s4} and
the induction hypothesis we conclude that
$${\rm sdepth}_{S_{i+1}}(J(G)^{(k)}\cap S_{i+1})\geq n-1-\nu_{o}(G\setminus N_G[x_{i+1}])\geq n-\nu_{o}(G).$$
Similarly$${\rm sdepth}_{S_{i+1}}(S_{i+1}/(J(G)^{(k)}\cap S_{i+1}))\geq n-\nu_{o}(G)-1.$$Now the claim follows by inequalities (\ref{3}), (\ref{4}), (\ref{5}), and (\ref{6}).

Now, $J_n'=(J(G)^{(k)}:x_1x_2\ldots x_n)$ and hence, Lemma \ref{colon} implies that $J_n'=J(G)^{(k-2)}$. Thus, by induction hypothesis we conclude that ${\rm sdepth}(J_n')\geq n-\nu_{o}(G)$ and ${\rm sdepth}(S/J_n')\geq n-\nu_{o}(G)-1$. Therefore, using the claim repeatedly implies that ${\rm sdepth}(J_1')\geq n-\nu_{o}(G)$ and ${\rm sdepth}(S/J_1')\geq n-\nu_{o}(G)-1$. This completes the proof of the theorem.
\end{proof}

As an immediate consequence of Theorems \ref{main} and \ref{ldepth}, we deduce the following result.

\begin{cor} \label{sineq}
Let $G$ be a graph and $J(G)$ be its edge ideal. Then for every integer $k\geq 2\nu_{o}(G)-1$, the modules $J(G)^{(k)}$ and $S/J(G)^{(k)}$ satisfy the Stanley's inequality.
\end{cor}

The following corollary is obtained by combining Theorems \ref{sdepthsym} and \ref{main}.

\begin{cor} \label{lim}
Let $G$ be a graph with $n$ vertices and $J(G)$ be its edge ideal. Then$$\lim_{k\rightarrow\infty}{\rm sdepth}(J(G)^{(k)})\geq n-\nu_{o}(G) \ \ \  {\rm and} \ \ \  \lim_{k\rightarrow\infty}{\rm sdepth}(S/J(G)^{(k)})\geq n-\nu_{o}(G)-1.$$
\end{cor}


\section*{Acknowledgment}
The author thanks the referee for careful reading of the paper and for helpful comments.



\end{document}